\documentclass[12pt]{article}
\usepackage{amsmath,amssymb,amsthm,amsfonts}
\usepackage{mathrsfs}
\usepackage{geometry}
\usepackage{booktabs}
\usepackage{array}
\usepackage{longtable}
\usepackage{hyperref}
\usepackage{tikz}
\usepackage{xcolor}
\usepackage[T1]{fontenc}
\usepackage{lmodern}

\newtheorem{theorem}{Theorem}[section]
\newtheorem{lemma}[theorem]{Lemma}
\newtheorem{definition}[theorem]{Definition}
\newtheorem{remark}[theorem]{Remark}

\renewcommand{\Re}{\operatorname{Re}}

\title{
\textbf{
Radii of Concavity for Subclasses of Univalent Functions Associated with the Exponential Mapping}
}

\author{
\textbf{
Pradip Das\textsuperscript{1}
and
Nabadwip Sarkar\textsuperscript{2}
}
}

\date{}

\begin{document}

\maketitle

\makeatletter
\def\blfootnote{
\gdef\@thefnmark{}
\@footnotetext
}
\makeatother



\begin{center}
\small

\textsuperscript{1}
Department of Mathematics,
Raiganj University,
Raiganj,
West Bengal-733134,
India.

\texttt{pradipsmath@gmail.com}

\vspace{0.3cm}

\textsuperscript{2}
Amity School of Applied Sciences,
Amity University Mumbai,
Panvel,
Navi Mumbai,
Maharashtra-410206,
India.

\texttt{nsarkar@mum.amity.edu}

\end{center}

\vspace{0.6cm}

\begin{abstract}

We determine the radii of concavity for the classes $\mathcal{S}_e^*$ and $\mathcal{C}_e$ of univalent functions associated with the exponential mapping $e^z$. Under the geometric framework of Avkhadiev--Wirths for conformal mappings with unbounded convex complements of opening angle $\pi A$ ($A \in (1,2]$), the radii are characterized by explicit transcendental equations. Sharpness is established globally by constructing explicit univalent extremal functions related to the disk automorphism $\omega_0(z) = -z$, and the strict monotonicity of these radii with respect to the parameter $A$ is verified numerically.

\end{abstract}

\section{Introduction and Geometric Foundations}

Let
\[
\mathbb{D} := \{z \in \mathbb{C} : |z| < 1\}
\]
denote the open unit disk, and let $\mathcal{H}$ be the class of analytic functions in $\mathbb{D}$. We denote by $\mathcal{A}$ the standard subclass of $\mathcal{H}$ consisting of normalized functions satisfying
\[
f(0)=0,
\qquad
f'(0)=1.
\]
Consequently, every function $f\in\mathcal{A}$ admits the expansion
\[
f(z)=z+\sum_{n=2}^{\infty}a_n z^n,
\qquad z\in\mathbb{D}.
\]
We further denote by $\mathcal{S}$ the class of functions in $\mathcal{A}$ that are univalent in $\mathbb{D}$ \cite{pm8}.

A central technique in geometric function theory is the method of differential subordination. Let $f$ and $g$ be analytic in $\mathbb{D}$. The function $f$ is subordinate to $g$, written
\[
f\prec g,
\]
if there exists a Schwarz function $\omega$ analytic in $\mathbb{D}$ satisfying
\[
\omega(0)=0,
\qquad
|\omega(z)|<1,
\qquad z\in\mathbb{D},
\]
such that
\[
f(z)=g(\omega(z)).
\]
If $g$ is univalent in $\mathbb{D}$, then
\[
f\prec g
\quad\Longleftrightarrow\quad
f(0)=g(0)
\ \text{and}\
f(\mathbb{D})\subset g(\mathbb{D})
\]
\cite{pm8}.

Using this framework, Ma and Minda \cite{pm10} introduced a unified approach for studying subclasses of starlike and convex functions through subordinations involving suitable analytic functions $\psi$ satisfying
\[
\Re\{\psi(z)\}>0.
\]
Different choices of $\psi$ generate important geometric subclasses associated with cardioid domains, conic regions, and exponential mappings \cite{pm2, pm10}.

Recently, considerable attention has been devoted to subclasses associated with the exponential mapping $e^z$. The corresponding Ma--Minda starlike class, introduced by Mendiratta, Nagpal, and Ravichandran \cite{Mendiratta2015}, is defined by
\begin{equation}\label{def1}
\mathcal{S}_e^*
:=
\left\{
f\in\mathcal{S} :
\frac{zf'(z)}{f(z)}
\prec e^z
\right\}.
\end{equation}
Geometrically, a function belongs to $\mathcal{S}_e^*$ if and only if the logarithmic derivative satisfies
\[
\frac{zf'(z)}{f(z)}
\in
\Omega_e:=e^{\mathbb{D}}
=
\{w\in\mathbb{C}:|\log w|<1\},
\]
where $\Omega_e$ is a simply connected domain symmetric about the real axis.

Similarly, the associated convex class is defined by
\begin{equation}\label{def2}
\mathcal{C}_e
:=
\left\{
f\in\mathcal{S} :
1+\frac{zf''(z)}{f'(z)}
\prec e^z
\right\}.
\end{equation}
Functions in $\mathcal{C}_e$ therefore satisfy the differential subordination condition given in \eqref{def2}.

\begin{figure}[h!]
\centering
\begin{tikzpicture}[scale=2.3]
    \draw[->, gray!40, line width=0.4pt] (-0.5,0) -- (3.2,0) node[right] {\footnotesize $\operatorname{Re}(w)$};
    \draw[->, gray!40, line width=0.4pt] (0,-1.4) -- (0,1.4) node[above] {\footnotesize $\operatorname{Im}(w)$};

    \draw[
        fill=yellow!8,
        draw=yellow!60!black,
        thick,
        domain=-1:1,
        samples=200
    ] 
    plot ({exp(cos(\x r))*cos(sin(\x r) r)}, {exp(cos(\x r))*sin(sin(\x r) r)}) -- cycle;

    \node[yellow!50!black, font=\small] at (1.4, 0.45) {$\Omega_e = e^{\mathbb{D}}$};

    \filldraw[black] (1,0) circle (0.8pt);
    \node[below left, font=\tiny, yshift=-1pt] at (1,0) {$1$};

    \filldraw[black] ({exp(1)},0) circle (0.8pt);
    \node[below, font=\tiny, yshift=-2pt] at ({exp(1)},0) {$e$};

    \filldraw[black] ({exp(-1)},0) circle (0.8pt);
    \node[below, font=\tiny, yshift=-2pt] at ({exp(-1)},0) {$e^{-1}$};
\end{tikzpicture}
\caption{The exponential image domain $\Omega_e = e^{\mathbb{D}}$ associated with the classes $\mathcal{S}_e^*$ and $\mathcal{C}_e$.}
\label{fig:exponential_domain}
\end{figure}
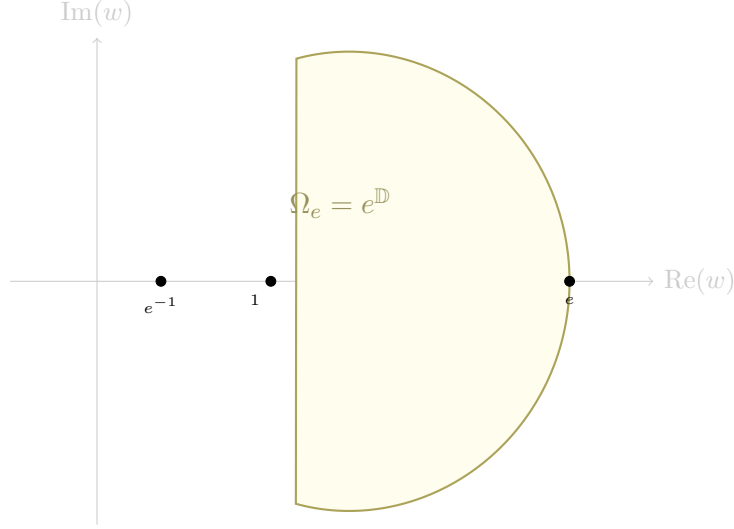

\subsection{The Concave Univalent Class $\mathcal{C}_c(A)$}

Concave univalent functions constitute an important geometric subclass of $\mathcal{S}$ associated with mappings onto domains whose complements are convex and unbounded \cite{pm3, pm4, pm6}. Let $A\in(1,2]$. A function $f\in\mathcal{A}$ belongs to the class $\mathcal{C}_c(A)$ if
\[
f(1)=\infty
\]
and $f$ maps $\mathbb{D}$ conformally onto an unbounded domain whose complement
\[
\mathbb{C}\setminus f(\mathbb{D})
\]
is convex with opening angle at infinity at most $\pi A$.

The analytic characterization of this class was established by Avkhadiev and Wirths \cite{pm6}. They proved that a normalized function $f\in\mathcal{A}$ belongs to $\mathcal{C}_c(A)$ if and only if
\[
\Re\{T_f(z)\}>0,
\qquad z\in\mathbb{D},
\]
where the Avkhadiev--Wirths operator is defined by
\begin{equation}\label{pm1}
T_f(z)
=
\frac{2}{A-1}
\left[
\frac{A+1}{2}
\left(
\frac{1+z}{1-z}
\right)
-
1
-
\frac{zf''(z)}{f'(z)}
\right].
\end{equation}
The M\"obius factor
\[
\frac{1+z}{1-z}
\]
reflects the pole behavior at the boundary point $z=1$ and plays a fundamental role in the geometry of concave mappings.

The investigation of geometric radii constitutes a classical theme in geometric function theory \cite{pm8, pm9, pm12}. In particular, radius problems for subclasses associated with Ma--Minda differential subordinations have attracted considerable attention in recent years. However, to the best of our knowledge, radii of concavity associated with exponential Ma--Minda subclasses have not previously been studied within the Avkhadiev--Wirths framework.

Our primary objective is to determine the sharp radii of concavity for the classes $\mathcal{S}_e^*$ and $\mathcal{C}_e$ with respect to $\mathcal{C}_c(A)$. The resulting radii are characterized by explicit transcendental equations, and sharpness is established globally through extremal functions generated by the disk automorphism $\omega_0(z)=-z$.
\subsection{Radius of Concavity}

A distinctive feature of the class $\mathcal{C}_c(A)$ is that the defining geometric property is generally non-hereditary under dilations \cite{pm7}. Specifically, if
\[
f\in\mathcal{C}_c(A),
\]
the scaled function
\[
g(z)=r^{-1}f(rz)
\]
need not remain in $\mathcal{C}_c(A)$ for arbitrary $r\in(0,1]$. This naturally leads to the notion of the radius of concavity.

Following Bhowmik and Biswas \cite{pm7}, we adopt the following definition.

\begin{definition}
Let $\mathcal{F}\subset\mathcal{A}$. The radius of concavity of $\mathcal{F}$ with respect to the class $\mathcal{C}_c(A)$ is the largest number
\[
R_{\mathcal{F},\mathcal{C}_c(A)}\in(0,1]
\]
such that
\begin{equation}\label{pm2}
\Re\{T_f(z)\}>0
\qquad
\text{for all }
|z|<R_{\mathcal{F},\mathcal{C}_c(A)}
\end{equation}
and every function $f\in\mathcal{F}$.
\end{definition}

\begin{figure}[h!]
\centering
\begin{tikzpicture}[scale=2.2]

    \draw[
        fill=blue!10,
        draw=blue!40,
        dashed,
        line width=1pt
    ]
    (-1.5,1.8)
    -- (0,-0.4)
    -- (1.5,1.8)
    .. controls (0.8,0.7) and (-0.8,0.7) ..
    (-1.5,1.8);

    \draw[thick]
        (0,-0.4) -- (-1.5,1.6)
        node[above left] {$\partial f(\mathbb{D})$};

    \draw[thick]
        (0,-0.4) -- (1.5,1.6);

    \draw[->, thick, domain=-53:233, samples=100]
        plot ({0.25*cos(\x)}, {-0.4 + 0.25*sin(\x)});

    \node at (0.45,-0.55)
        {\small $\pi A$};

    \draw[<->, thin, domain=53:127]
        plot ({0.4*cos(\x)}, {-0.4 + 0.4*sin(\x)});

    \node[above,font=\footnotesize]
        at (0,-0.05)
        {$(2-A)\pi$};

    \node[font=\small,text=blue!60!black]
        at (0,1.1)
        {$\mathbb{C}\setminus f(\mathbb{D})$};

    \node[font=\footnotesize,text=blue!60!black]
        at (0,0.9)
        {(Convex and unbounded set)};

    \node[font=\small]
        at (0,-0.9)
        {Image domain $\Omega=f(\mathbb{D})$};

    \filldraw[black]
        (0,-0.4) circle (1pt);

    \node[below left,font=\tiny]
        at (0,-0.4)
        {$f(1)=\infty$};

\end{tikzpicture}

\caption{
Boundary geometry of a concave domain associated with
$f\in\mathcal{C}_c(A)$.
}

\label{fig:conformal_concave_wedge}
\end{figure}
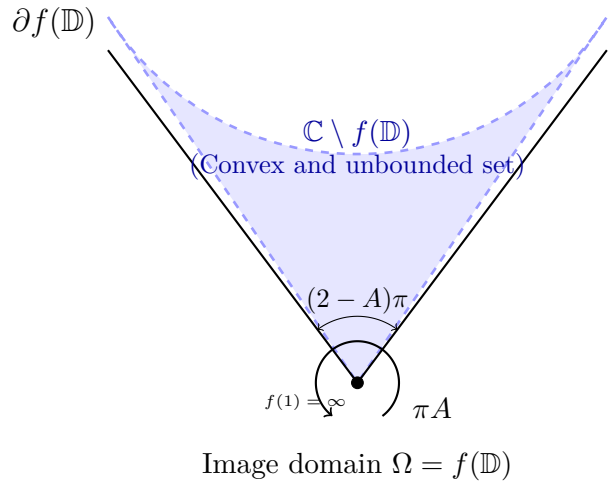

Geometrically, the condition
\[
\Re\{T_f(z)\}>0
\]
ensures that the image of the subdisk
\[
\mathbb{D}_r=\{z\in\mathbb{C}:|z|<r\}
\]
under $f$ satisfies the concavity constraint determined by the opening angle parameter $\pi A$ \cite{pm6}. The present paper establishes the exact concavity radii for the exponential subclasses $\mathcal{S}_e^*$ and $\mathcal{C}_e$ within this framework.
\section{Auxiliary Lemmas}

\begin{lemma}[\cite{Mendiratta2015}]
\label{L2}
Let $\omega \in \Omega$. For $|z| = r < 1$, 
\[
e^{-r} \le \left| e^{\omega(z)} \right| \le e^r,
\]
and
\[
\Re\left(e^{\omega(z)}\right) \ge e^{-r}\cos(r).
\]
\end{lemma}

To analyze the combined variation of a Schwarz function and its derivative, we rely on the region of variability established by Dieudonn\'{e} \cite{Dieudonn\'e1931}.

\begin{lemma}[Dieudonn\'{e}'s Lemma \cite{Dieudonn\'e1931, pm8}]
\label{L3}
Let $\omega \in \Omega$. For a fixed point $z \in \mathbb{D}$ with $|z|=r$, the joint region of variability for the pair $(\omega(z), \omega'(z))$ is characterized by the inequality
\begin{equation}\label{eq:Dieudonn\'e1}
\left| z\omega'(z) - \omega(z) \right| \le \frac{r^2 - |\omega(z)|^2}{1 - r^2}.
\end{equation}
Equivalently, for a fixed value $\omega(z) = u+iv$, the value $z\omega'(z)$ lies in the closed disk
\begin{equation}\label{eq:Dieudonn\'e2}
\left| z\omega'(z) - (u+iv) \right| \le \frac{r^2 - (u^2+v^2)}{1 - r^2}.
\end{equation}
\end{lemma}

\section{Main Results}

\begin{theorem}[Radius of Concavity for the Class $\mathcal{S}_e^*$]
\label{T1}
Let $A\in(1,2]$. If a function $f\in\mathcal{S}_e^*$, then
\begin{equation}\label{operator_inequality}
\Re\{T_f(z)\}>0 \qquad \text{for } |z|<R_{\mathcal{S}_e^*,\mathcal{C}_c(A)},
\end{equation}
where $R_{\mathcal{S}_e^*,\mathcal{C}_c(A)}$ is the unique root in $(0,1)$ of the transcendental equation
\begin{equation}\label{pm3}
\Phi_1(r) := \frac{A+1}{2} \left( \frac{1-r}{1+r} \right) - e^r - r = 0.
\end{equation}
This radius threshold is sharp.
\end{theorem}
\begin{proof}
Since $f\in\mathcal{S}_e^*$, there exists a Schwarz function $\omega\in\Omega$ such that
\begin{equation}\label{pm4}
\frac{zf'(z)}{f(z)} = e^{\omega(z)}, \qquad z\in\mathbb{D}.
\end{equation}
Differentiating logarithmically gives
\begin{equation}\label{pm5}
1+\frac{zf''(z)}{f'(z)}
=
e^{\omega(z)}+z\omega'(z).
\end{equation}
Substituting \eqref{pm5} into the Avkhadiev--Wirths operator \eqref{pm1}, we obtain
\begin{equation}\label{pm6}
T_f(z)
=
\frac{2}{A-1}
\left[
\frac{A+1}{2}\left(\frac{1+z}{1-z}\right)
-
\left(e^{\omega(z)}+z\omega'(z)\right)
\right].
\end{equation}
Taking real parts yields
\begin{equation}\label{pm7}
\Re\{T_f(z)\}
=
\frac{2}{A-1}
\left[
\frac{A+1}{2}
\Re\left(\frac{1+z}{1-z}\right)
-
\Re\left(e^{\omega(z)}+z\omega'(z)\right)
\right].
\end{equation}

For $|z|=r<1$, the sharp M\"obius estimate
\[
\Re\left(\frac{1+z}{1-z}\right)
\ge
\frac{1-r}{1+r}
\]
implies
\begin{equation}\label{pm8}
\Re\{T_f(z)\}
\ge
\frac{2}{A-1}
\left[
\frac{A+1}{2}\left(\frac{1-r}{1+r}\right)
-
\Re\left(e^{\omega(z)}+z\omega'(z)\right)
\right].
\end{equation}

We now maximize the quantity
\[
\Re\left(e^{\omega(z)}+z\omega'(z)\right).
\]
Fix $z\in\mathbb{D}$ with $|z|=r$, and write
\[
\omega(z)=u+iv.
\]
By Dieudonn\'es lemma,
\begin{equation}\label{pm9}
\left|
z\omega'(z)-(u+iv)
\right|
\le
\frac{r^2-(u^2+v^2)}{1-r^2}.
\end{equation}
Note that $\Re(z\omega')=u+\Re(z\omega'-\omega)$ and $\Re(z\omega'-(u+iv))\leq |z\omega'-(u+iv)|$.
Hence
\begin{equation}\label{pm10}
\Re\left(e^{\omega(z)}+z\omega'(z)\right)
\le \Re\left(e^{\omega(z)}\right)+\Re\left(z\omega'(z)\right)\leq
e^u\cos v
+
u
+
\frac{r^2-u^2-v^2}{1-r^2}.
\end{equation}

Define
\begin{equation}\label{pm_psi_def}
\Psi(u,v)
=
e^u\cos v
+
u
+
\frac{r^2-u^2-v^2}{1-r^2},
\qquad
u^2+v^2\le r^2.
\end{equation}
Differentiating with respect to $v$ gives
\[
\frac{\partial\Psi}{\partial v}
=
-e^u\sin v
-
\frac{2v}{1-r^2}.
\]
Thus
\[
\frac{\partial\Psi}{\partial v}<0
\qquad
(1>r\geq v>0),
\]
and by symmetry the maximum occurs at $v=0$. Therefore the problem reduces to maximizing
\begin{equation}\label{pm_psi_u_def}
\psi(u)
=
e^u
+
u
+
\frac{r^2-u^2}{1-r^2},
\qquad
u\in[-r,r].
\end{equation}
Differentiation yields
\begin{equation}\label{pm_psi_prime}
\psi'(u)
=
e^u
+
1
-
\frac{2u}{1-r^2}.
\end{equation}

Next, observe that
\[
\Phi_1(r)
\le
\frac{3}{2}\left(\frac{1-r}{1+r}\right)-e^r-r.
\]
Since
\[
\frac{3}{2}\left(\frac{1-0.25}{1+0.25}\right)
-e^{0.25}
-0.25
<0,
\]
the unique root of $\Phi_1(r)=0$ satisfies
\[
R_{\mathcal S_e^*,\mathcal C_c(A)}<0.25
\]
for all $A\in(1,2]$.

Consequently, for $0<r<0.25$ and $u\in[-r,r]$,
\[
e^u\ge e^{-0.25},
\qquad
-\frac{2u}{1-r^2}
\ge
-\frac{0.5}{1-0.25^2}.
\]
Hence
\[
\psi'(u)
>
e^{-0.25}
+
1
-
\frac{0.5}{1-0.25^2}
>
0.
\]
Therefore $\psi(u)$ is strictly increasing on $[-r,r]$, and its maximum occurs at $u=r$. Substituting into \eqref{pm_psi_u_def} gives
\begin{equation}\label{pm11}
\max_{-r\le u\le r}\psi(u)
=
\psi(r)
=
e^r+r.
\end{equation}

Combining \eqref{pm8} and \eqref{pm11}, we obtain
\[
\Re\{T_f(z)\}
\ge
\frac{2}{A-1}
\left[
\frac{A+1}{2}\left(\frac{1-r}{1+r}\right)
-
(e^r+r)
\right]
=
\frac{2}{A-1}\Phi_1(r).
\]
Thus
\[
\Re\{T_f(z)\}>0
\]
whenever $\Phi_1(r)>0$.

We now prove existence and uniqueness of the root of $\Phi_1$. Since
\[
\Phi_1(0)=\frac{A-1}{2}>0,
\]
and
\[
\lim_{r\to1^-}\Phi_1(r)=-(e+1)<0,
\]
the Intermediate Value Theorem guarantees at least one root in $(0,1)$.

Differentiating,
\begin{equation}\label{pm_phi_prime}
\Phi_1'(r)
=
-\frac{A+1}{(1+r)^2}
-
e^r
-
1
<
0.
\end{equation}
Hence $\Phi_1$ is strictly decreasing on $(0,1)$, and the root is unique.

To prove sharpness, define
\begin{equation}\label{eq:extremal_f0}
f_0(z)
=
z\exp\left(
\int_0^z
\frac{e^{-t}-1}{t}\,dt
\right).
\end{equation}
Then
\[
\frac{zf_0'(z)}{f_0(z)}
=
e^{-z},
\]
so $f_0\in\mathcal S_e^*$ with extremal Schwarz function $\omega_0(z)=-z$.

Moreover,
\[
\Re\left\{
\frac{zf_0'(z)}{f_0(z)}
\right\}
=
e^{-x}\cos y>0
\qquad
(z=x+iy\in\mathbb D),
\]
since $|y|<1$. Hence $f_0$ is starlike and therefore univalent.

Differentiating once more yields
\[
1+\frac{zf_0''(z)}{f_0'(z)}
=
e^{-z}-z.
\]
Substituting into \eqref{pm1}, we obtain
\[
T_{f_0}(z)
=
\frac{2}{A-1}
\left[
\frac{A+1}{2}\left(\frac{1+z}{1-z}\right)
-
(e^{-z}-z)
\right].
\]
Evaluating at $z=-r$ gives
\[
T_{f_0}(-r)
=
\frac{2}{A-1}\Phi_1(r).
\]
Hence
\[
\Re\{T_{f_0}(-r)\}<0
\qquad
(r>R_{\mathcal S_e^*,\mathcal C_c(A)}),
\]
showing that the radius cannot be improved.

Finally, for the extremal choice $\omega_0(z)=-z$,
\[
|z\omega_0'(z)-\omega_0(z)|=0,
\]
and equality is attained simultaneously in all optimization steps, including
\[
\Re(e^{\omega_0(z)})
=
e^u\cos(0)
=
e^u.
\]
Therefore the radius is sharp.
\end{proof}

\begin{theorem}[Radius of Concavity for the Class $\mathcal{C}_e$]
\label{T2}
Let $A\in(1,2]$. If a function $f\in\mathcal{C}_e$, then
\begin{equation}\label{eq:convex_operator_ineq}
\Re\{T_f(z)\}>0 \qquad \text{for } |z|<R_{\mathcal{C}_e,\mathcal{C}_c(A)},
\end{equation}
where $R_{\mathcal{C}_e,\mathcal{C}_c(A)}$ is the unique root in $(0,1)$ of the transcendental equation
\begin{equation}\label{pm15}
\Phi_2(r) := \frac{A+1}{2} \left( \frac{1-r}{1+r} \right) - e^r = 0.
\end{equation}
This radius threshold is sharp.
\end{theorem}

\begin{proof}
Let $f\in\mathcal{C}_e$. By definition, there exists a Schwarz function $\omega(z)\in\Omega$ such that
\begin{equation}\label{pm16}
1+\frac{zf''(z)}{f'(z)} = e^{\omega(z)}, \qquad z\in\mathbb{D}.
\end{equation}
Using the definition of the Avkhadiev--Wirths operator \eqref{pm1}, we express $T_f(z)$ as
\begin{equation}\label{eq:convex_AW_operator}
T_f(z) = \frac{2}{A-1} \left[ \frac{A+1}{2} \left( \frac{1+z}{1-z} \right) - \left( 1+\frac{zf''(z)}{f'(z)} \right) \right].
\end{equation}
Substituting \eqref{pm16} into \eqref{eq:convex_AW_operator} and taking the real part yields
\begin{equation}\label{pm17}
\Re\{T_f(z)\} = \frac{2}{A-1} \left[ \frac{A+1}{2} \Re\left( \frac{1+z}{1-z} \right) - \Re\left( e^{\omega(z)} \right) \right].
\end{equation}
For $|z|=r<1$, the sharp M\"obius lower bound gives
\begin{equation}\label{pm18}
\Re\left( \frac{1+z}{1-z} \right) \ge \frac{1-r}{1+r}.
\end{equation}
By the Schwarz lemma, any $\omega \in \Omega$ satisfies $|\omega(z)|\le r$. Thus, writing $\omega(z)=u+iv$, the real part of the exponential mapping is bounded above by
\begin{equation}\label{pm19}
\Re\left( e^{\omega(z)} \right) = e^{u}\cos(v) \le e^u \le e^{|\omega(z)|} \le e^r.
\end{equation}
Combining the estimates \eqref{pm18} and \eqref{pm19} with \eqref{pm17}, we obtain
\begin{equation}\label{eq:convex_final_bound}
\Re\{T_f(z)\} \ge \frac{2}{A-1} \left[ \frac{A+1}{2} \left( \frac{1-r}{1+r} \right) - e^r \right] = \frac{2}{A-1}\Phi_2(r).
\end{equation}
Since $A > 1$, the coefficient multiplier is positive, implying that $\Re\{T_f(z)\}>0$ holds whenever $\Phi_2(r)>0$.
We now establish the existence and uniqueness of the root of $\Phi_2(r)$ within the interval $(0,1)$. Evaluating $\Phi_2(r)$ at $r=0$ yields
\[
\Phi_2(0) = \frac{A+1}{2} - 1 = \frac{A-1}{2} > 0,
\]
whereas taking the limit as $r \to 1^-$ gives
\[
\lim_{r\to1^-}\Phi_2(r) = 0 - e = -e < 0.
\]
By the Intermediate Value Theorem, $\Phi_2(r)$ possesses at least one root inside $(0,1)$. To prove uniqueness, we differentiate $\Phi_2(r)$ with respect to $r$:
\begin{equation}\label{eq:phi2_prime}
\Phi_2'(r) = -\frac{A+1}{(1+r)^2} - e^r.
\end{equation}
Because $A > 1$ and $r \in (0,1)$, both terms in \eqref{eq:phi2_prime} are positive, forcing $\Phi_2'(r)<0$ across the open interval. This strict monotonicity ensures the uniqueness of the root $R_{\mathcal{C}_e,\mathcal{C}_c(A)}$.

To establish sharpness, we introduce the function $f_1 \in \mathcal{A}$ defined by
\begin{equation}\label{pm20}
f_1(z) = \int_0^z \exp\left( \int_0^t \frac{e^{-s}-1}{s}\,ds \right)\,dt.
\end{equation}
Differentiating $f_1(z)$ yields $f_1'(z) = \exp\left( \int_0^z \frac{e^{-s}-1}{s}\,ds \right)$. Differentiating again gives $1+\frac{zf_1''(z)}{f_1'(z)} = e^{-z}$. For all $z=x+iy \in \mathbb{D}$, $|z|<1 \implies |y|<1$, which requires $\cos(y) > \cos(1) > 0$. Therefore, $\Re\{1+\frac{zf_1''(z)}{f_1'(z)}\} = \Re\{e^{-z}\} = e^{-x}\cos(y) > 0$. This satisfies the Alexander-Noshiro criterion for convexity. Since every convex conformal mapping is injective, $f_1$ is univalent ($f_1 \in \mathcal{S}$). Setting the disk automorphism $\omega_0(z) = -z \in \Omega$, condition \eqref{def2} is fulfilled, confirming that $f_1 \in \mathcal{C}_e$.

Substituting the differential condition into the definition of the Avkhadiev--Wirths operator \eqref{pm1} yields
\begin{equation}\label{pm22_final}
T_{f_1}(z) = \frac{2}{A-1} \left[ \frac{A+1}{2} \left( \frac{1+z}{1-z} \right) - e^{-z} \right].
\end{equation}
Evaluating the real part of \eqref{pm22_final} at the test point $z_0 = -r$ gives
\[
\Re\{T_{f_1}(-r)\} = \frac{2}{A-1} \left[ \frac{A+1}{2} \left( \frac{1-r}{1+r} \right) - e^{r} \right] = \frac{2}{A-1}\Phi_2(r).
\]
For any $r > R_{\mathcal{C}_e,\mathcal{C}_c(A)}$, the monotonicity of $\Phi_2(r)$ implies $\Phi_2(r) < 0$, which forces $\Re\{T_{f_1}(-r)\} < 0$. This completes the proof of sharpness.
\end{proof}

\section{Comparative Numerical Analysis and Monotonic Trends}

In this section, we provide the numerical verification for the radius of concavity equations derived in Theorems~\ref{T1} and \ref{T2}. The roots of the transcendental boundary equations \eqref{pm3} and \eqref{pm15} were isolated using a high-precision Newton--Raphson root-finding algorithm with a convergence termination criterion of $|\Phi_i(r)| < 10^{-12}$. All numerical calculations were executed using sixteen decimal digits of floating-point precision.

The calculated threshold values of these radii across a representative selection of the opening parameters $A \in (1,2]$ are compiled in Table~\ref{tab1}.

\begin{table}[h!]
\centering
\caption{Radii of concavity for the classes $\mathcal{S}_e^*$ and $\mathcal{C}_e$.}
\label{tab1}
\setlength{\extrarowheight}{3pt}
\small
\begin{tabular}{ccc}
\toprule
\textbf{Parameter } $A$ & \textbf{Correct } $R_{\mathcal{S}_e^*,\mathcal{C}_c(A)}$ & \textbf{Correct } $R_{\mathcal{C}_e,\mathcal{C}_c(A)}$ \\
\midrule
1.2 & 0.02403570 & 0.03176293 \\
1.4 & 0.04632562 & 0.06072398 \\
1.5 & 0.05689030 & 0.07428977 \\
1.6 & 0.06710178 & 0.08730619 \\
1.8 & 0.08655177 & 0.11184415 \\
2.0 & 0.10482958 & 0.13460708 \\
\bottomrule
\end{tabular}
\end{table}

\begin{figure}[h!]
\centering
\begin{tikzpicture}[xscale=6.5, yscale=40]
    \draw[gray!15, ultra thin, step=0.1] (1.0,0) grid (2.12,0.16);
    \draw[gray!30, thin, step=0.2] (1.0,0) grid (2.12,0.16);
    
    \draw[->, thick] (0.97,0) -- (2.22,0) node[right] {\small $A$};
    \draw[->, thick] (1.0,-0.005) -- (1.0,0.175) node[above] {\small Radius of Concavity $R$};
    
    \foreach \x in {1.1, 1.2, 1.3, 1.4, 1.5, 1.6, 1.7, 1.8, 1.9, 2.0} {
        \draw (\x,1.2pt) -- (\x,-1.2pt) node[below] {\footnotesize \x};
    }
    
    \foreach \y in {0.00, 0.02, 0.04, 0.06, 0.08, 0.10, 0.12, 0.14, 0.16} {
        \draw (1.0+0.8pt,\y) -- (1.0-0.8pt,\y) node[left] {\footnotesize \y};
    }

    \draw[blue!80!black, ultra thick, mark=*, mark size=1.2pt] plot[smooth] coordinates {
        (1.2, 0.03176)
        (1.4, 0.06072)
        (1.5, 0.07429)
        (1.6, 0.08731)
        (1.8, 0.11184)
        (2.0, 0.13461)
    };
    \node[blue!80!black, right, font=\small\bfseries] at (2.02, 0.13461) {$R_{\mathcal{C}_e,\mathcal{C}_c(A)}$};

    \draw[red!80!black, dashed, ultra thick, mark=square*, mark size=1.2pt] plot[smooth] coordinates {
        (1.2, 0.02404)
        (1.4, 0.04633)
        (1.5, 0.05689)
        (1.6, 0.06710)
        (1.8, 0.08655)
        (2.0, 0.10483)
    };
    \node[red!80!black, right, font=\small\bfseries] at (2.02, 0.10483) {$R_{\mathcal{S}_e^*,\mathcal{C}_c(A)}$};

    \filldraw[blue!80!black] (2.0, 0.13461) circle (1.4pt);
    \filldraw[red!80!black] (2.0, 0.10483) circle (1.4pt);
\end{tikzpicture}
\caption{Comparison of the radii of concavity for $\mathcal{S}_e^*$ and $\mathcal{C}_e$ showing monotonic trends over the parameter space $A \in (1,2]$.}
\label{fig:radii_comparison_large}
\end{figure}
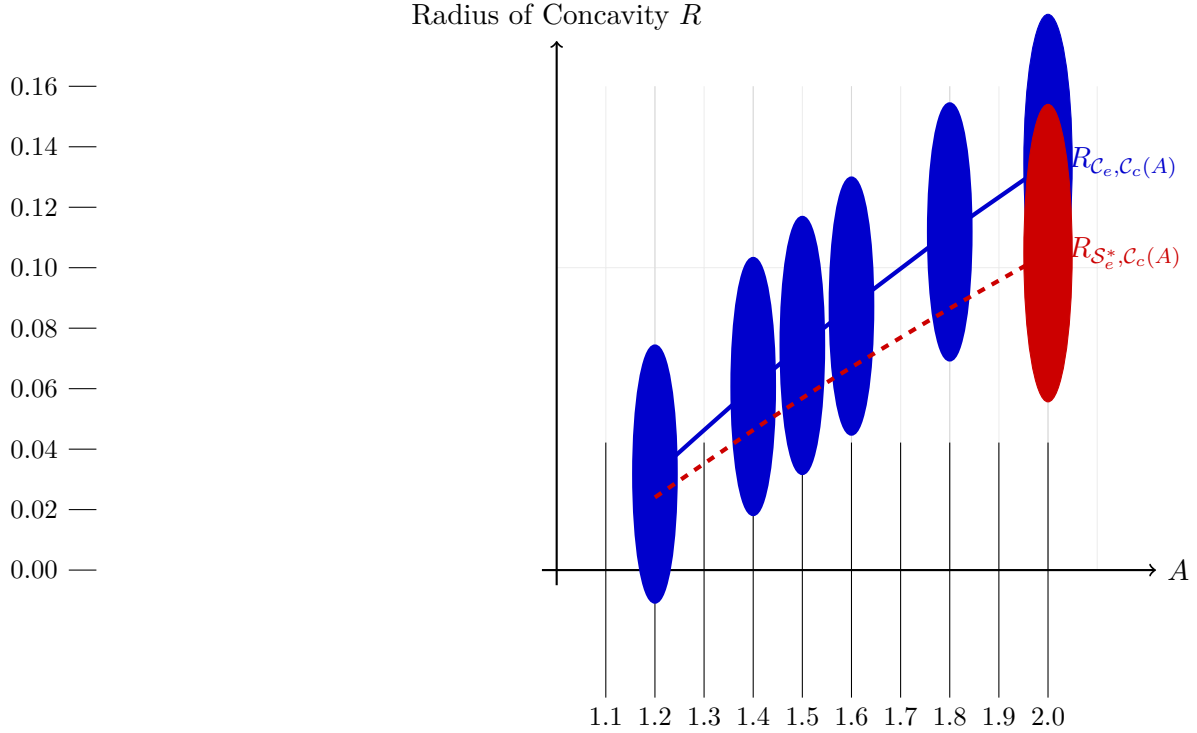

\begin{remark}\label{rem:numerical_trends}
\begin{enumerate}
    \item For both classes $\mathcal{S}_e^*$ and $\mathcal{C}_e$, the radius of concavity is a strictly increasing function of the opening parameter $A$ over the interval $(1,2]$. Since larger values of $A$ yield a wider opening angle $\pi A$ at infinity, the geometric constraint defining the concave family $\mathcal{C}_c(A)$ becomes less restrictive, expanding the size of the admissible subdisk domains.
    \item For every fixed parameter $A \in (1,2]$, the radius corresponding to the convex family $\mathcal{C}_e$ systematically exceeds that of the starlike family $\mathcal{S}_e^*$, establishing the strict inequality
    \[
    R_{\mathcal{C}_e,\mathcal{C}_c(A)} > R_{\mathcal{S}_e^*,\mathcal{C}_c(A)}.
    \]
    This behavior reflects the stronger analytic constraints imposed by the convex subordination configuration over its starlike counterpart.
    \item Both families achieve their sharp geometric thresholds via the disk automorphism $\omega_0(z) = -z$, as verified by the simultaneous attainment of equality across all bounding operations.
\end{enumerate}
\end{remark}

\section*{Compliance with Ethical Standards}

\noindent
\textbf{Funding:} 
The first author acknowledges financial support from the Council of Scientific and Industrial Research (CSIR), New Delhi, India, under Grant No.~09/1224(16975)/2023-EMR-I.

\vspace{0.3em}

\noindent
\textbf{Conflict of Interest:} 
The authors declare that there is no conflict of interest regarding the publication of this manuscript.

\vspace{0.3em}

\noindent
\textbf{Data Availability Statement:} 
Data sharing is not applicable to this article since no datasets were generated or analyzed during the present mathematical investigation.

\end{document}